\documentclass[12pt]{amsart}

\usepackage{amsmath,amsfonts,amssymb,amsthm,epsfig,epstopdf,url,array,secdot}

\usepackage{graphicx}

\theoremstyle{plain}

\newtheorem{definition}{Definition}[section]
\newtheorem{theorem}{Theorem}[section]
\newtheorem{lemma}[theorem]{Lemma}

\newtheorem{corollary}[theorem]{Corollary}

\title{Well-quasi-ordering in lattice path matroids}

\author[Jose]{Meenu Mariya Jose}
\address{School of Mathematics and Statistics,
	Victoria University of Wellington,
	New Zealand}
\email{josemeen@myvuw.ac.nz}

\author[Mayhew]{Dillon Mayhew}
\address{School of Mathematics and Statistics,
	Victoria University of Wellington,
	New Zealand}
\email{dillon.mayhew@vuw.ac.nz}

\date{\today}


\begin{document}
	
	\begin{abstract}
		Lattice path matroids form a subclass of transversal matroids and were introduced by Bonin, de Mier and Noy \cite{Latt1}. Transversal matroids are not well-quasi-ordered, even when the branch-width is restricted. Though lattice path matroids are not well-quasi-ordered, we prove that lattice path matroids of bounded branch-width are well-quasi-ordered.
	\end{abstract}
	
	\maketitle
	
	\section{Introduction}
	
	Well-quasi-ordering is at the heart of major projects undertaken in discrete mathematics in the recent years. A \emph{quasi-ordering} is a relation that is reflexive and transitive. A \emph{well-quasi-ordering} is a quasi-ordering with the property that if $ a_{0}, a_{1},\ldots $ is an infinite sequence in the set $ A $ that we consider, then there exists $ i $ and $ j $ such that $ a_{i} \leq a_{j} $. 
	
	The expansive Robertson-Seymour graph-minors project was a major accomplishment in discrete mathematics. It was a gigantic feat spanning over 500 pages that proved that any infinite collection of graphs is well-quasi-ordered under the minor relation.
	
	Recently, Geelen, Gerards and Whittle announced a proof (Theorem 6, \cite{article}) that the class of $ F $-representable matroids is well-quasi-ordered under the minor relation, where $ F $ is any finite field. This is connected to their proof of Rota's conjecture. One of the crucial steps that edged them closer to the proof of the former conjecture was their 2002 proof \cite{GEELEN} that a class of $ F $-representable matroids with bounded branch-width is well-quasi-ordered. We will accomplish the same goal for the class of lattice path matroids.

	The class of lattice path matroids is an elegant class discovered by Bonin, de Mier and Noy \cite{Latt1} with nice structural properties. It is a subclass of transversal matroids, but surprisingly is closed under duality. Interestingly enough, lattice path matroids are not well-quasi-ordered (for examples, refer to Section \ref{wqo}). However, we prove that they are well-quasi-ordered when certain restrictions are placed on the class. 
	
	\begin{theorem}
		Lattice path matroids of bounded branch-width are well-quasi-ordered.
	\end{theorem} 
	
	The proof of the theorem uses the elegant minimal bad sequence argument that Nash-Williams employs to prove that finite trees are well-quasi-ordered \cite{MR0153601}.
	
	On the other hand, transversal matroids do not behave so well under those same restrictions. To observe this, consider the well-known polygon matroid anti-chain (Example 14.1.2, \cite{book}) that starts with the matroids represented in Fig \ref{poly} .
	
	\begin{figure}[h] \label{poly}
		\centering
		\includegraphics[scale=0.4]{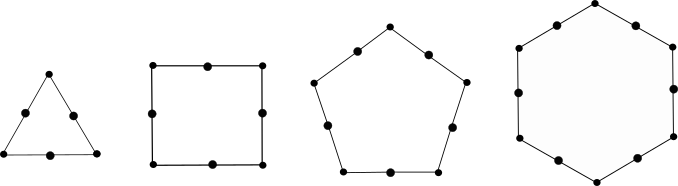}
		\caption{The first four matroids in the polygon matroid anti-chain}
		
	\end{figure}
	
	This is a class of rank-3 matroids, and hence bounded branch-width (at most 3). Since branch-width of the dual of a matroid is same as that of the original matroid, the dual class of this matroid class also has bounded branch-width. Since strict gammoids are exactly the duals of transversal matroids (Corollary 2.4.5, \cite{book}), it is enough to show that this infinite class indeed consists of strict gammoids, and this is an easy exercise. Thus we have found an infinite class of transversal matroids of bounded branch-width that form an anti-chain.\\

	\section{Lattice Path Matroids}
	
	The class of lattice path matroids can be better understood in terms of lattice paths.
	
	A lattice path in $ \mathbb{Z}^d $ of length $ n $ with steps in $ S $ is a sequence of points $ s_{0}s_{1}\ldots s_{n} $ such that each $ s_{i}-s_{i-1} $ is in $ S $. For the paths considered here, $ S = \{E, N\} $, which are called $ East $ and $ North $ respectively, where $ E = (1,0) $ (or moving right) and $ N = (0,1) $ (or moving up). The paths are written as words or strings in the alphabet $ \{E,N\} $.
	
	Let $ P $ and $ Q $ be two lattice paths. Then, $ P_{i} $ and $ Q_{i} $ are the sub-strings that represent the first $ i $ steps of the respective paths.
	Let $ P = p_{1}p_{2}\ldots p_{m+r} $ and $ Q = q_{1}q_{2}\ldots q_{m+r} $ be two lattice paths from $ (0,0) $ to $ (m,r) $ such that the paths either go East or North, where $ P $ never goes above $ Q $, or in other words, for every $ i $, the number of North steps in $ P_{i} $ is never more than that in $ Q_{i} $ (and the number of East steps in $ Q_{i} $ is never more than that in $ P_{i} $ ). Let $ {p_{u_{1}},p_{u_{2}},\ldots, p_{u_{r}}} $ be the set of North steps of $ P $  with $ u_{1}<u_{2}<\ldots<u_{r} $. Let $ {q_{l_{1}},q_{l_{2}},\ldots, q_{l_{r}}} $ be the set of North steps of $ Q $  with $ l_{1}<l_{2}<\ldots<l_{r} $. Let $ N_{i} $ be the interval $ [l_{i},u_{i}] $ of integers. Let $ M[P,Q] $ be the transversal matroid that has ground set $ [m+r] $ and presentation $ (N_{i}:i \in [r]) $. A $ lattice $ $path $ $ matroid $ is a transversal matroid that is isomorphic to $ M[P,Q] $ for some such pair of lattice path $ P $ and $ Q $.\\
	
	\begin{figure}[h] 
		\centering
		\includegraphics{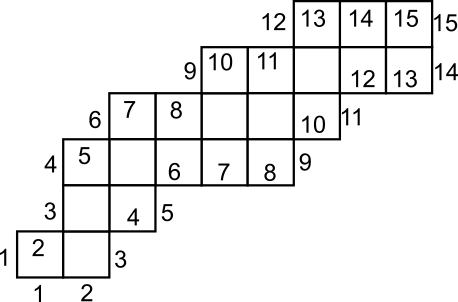} 
		\caption{A lattice path matroid that starts at $ (0,0) $ and ends at $ (9,6) $ }
	\end{figure}

	When thought of as arising from the particular presentation of bounding paths $ P $ and $ Q $, the elements are in their natural order. However this order is not evident in the matroid structure. \\

	Let $ X $ be a subset of the ground set $ [m+r] $ of the lattice path matroid $ M[P,Q] $. The lattice path $ P(X) $ is the word $ s_{1}s_{2}\ldots s_{m+r} $ in the alphabet $ \{E,N\} $, where  
	\begin{align*}
	s_{i}  &= \begin{cases} N, & \text{if} \quad i \in X \\ E, & \text{otherwise.} \end{cases}
	\end{align*}
	
	Hence independent sets are subsets of $ [m+r] $ such that $|I| \leq r $ and the associated path $ P(I) $ is a path or part of a path that lies between $ P $ and $ Q $. Alternatively, independent sets are partial transversals of the corresponding transversal matroid. This leads to the following characterisation of bases of lattice path matroids:\\
	
	A subset $ B $ of $ [m+r] $ with $ |B|=r $ is a basis of $ M[P,Q] $ if and only if the associated lattice path stays in the region bounded by $ P $ and $ Q $ (Theorem 3.3, \cite{Latt1}).\\

	\textbf{Minors:} Single element deletions and contractions can be described in terms of bounding paths of $ M = M[P,Q] $  as follows: An isthmus is an element $ x $ for which some $ N_{i} $ is $ \{x\} $. So, to delete or contract $ x $, delete the corresponding North step from both the bounding paths. Correspondingly, to delete or contract a loop, which is an element that is in no set $ N_{i} $, delete that East step from both the bounding paths.
	
	If $ x $ is neither a loop nor an isthmus, the upper bounding path of $ M\backslash x $ is formed by deleting from $ Q $ the first East step that is at or after $ x $ and the lower bounding path is obtained by deleting from $ P $  the last East step that is at or before $ x $. Dually, the upper bounding path of $ M/x $ is formed by deleting from $ Q $ the last North step that is at or before $ x $ and the lower bounding path is obtained by deleting from $ P $  the first North step that is at or after $ x $.

	Lattice path matroids are closed under minors, duals and direct sums (Theorem 3.1, \cite{BONIN}) but is not closed under the operations of truncation, elongation and free extension.
	
	Nested matroids form a subclass of lattice path matroids that has appeared under different names in varying contexts. A nested matroid is obtained from the empty matroid by iterating the operations of adding co-loops and taking free extensions to the empty matroid. Bonin and de Mier \cite{BONIN} defined them in terms of lattice path matroids as a matroid of the form $ M[P,Q] $, where $ P = E^{m}N^{r} $ and named them generalised Catalan matroids. They later proved that nested matroids are well-quasi-ordered (Theorem 5.4, \cite{Latt2}).
	
	\section{Well-Quasi-Ordering} \label{wqo}
	
	An infinite sequence $ a_{1},a_{2},\ldots $ is called an anti-chain when there does not exist $ i < j $ such that $ a_{i} \leq a_{j} $. Also, a sequence $ a_{1},a_{2},\ldots $ is infinitely strictly decreasing if $ a_{i} > a_{i+1} $ for $ i \geq 1 $. Equivalently, a quasi-order is a well-quasi-order is and only if there exists no infinite anti-chain or infinite decreasing sequence (See, for example, \cite{citeulike:395714}, Prop 12.1.1). 
	
	As we mentioned in the introduction, lattice path matroids are not well-quasi-ordered. There is a subclass of lattice path matroids named notch matroids by Bonin and de Mier. Their paper includes an excluded minor characterisation for notch matroids (Theorem 8.8, \cite{BONIN}). Among the excluded minors are three infinite families of lattice path matroids, which are listed below:

	\begin{itemize}
		
		\item for $ n \geq 4 $, the rank-$n$ matroid $ F_{n} := T_{n}(U_{n-2,n-1} \oplus U_{n-2,n-1}) $,
		\item for $ n \geq 2 $, the rank-$ n $ matroid $ G_{n} := T_{n}(U_{n-1,n+1} \oplus U_{n-1,n+1}) $, and
		\item  for $ n \geq 3 $, the rank-$ n $ matroid $ H_{n} := T_{n}(U_{n-2,n-1} \oplus U_{n-1,n+1}) $, 
	\end{itemize}
	
	where $ T_{n} $ denotes the truncation to rank $ n $. Thus we conclude that these infinite families each form an anti-chain in the class of lattice path matroids and hence the class is not well-quasi-ordered.

	\section{Square-width}

	Let $ [P,Q] $ be a pair of lattice paths that correspond to the lattice path matroid $ M[P,Q] $. We say that $ [P,Q]  $ is a $ path $  $ presentation $ that corresponds to the matroid $ M[P,Q] $. The size of a presentation is nothing but the size of the ground set of the corresponding matroid. We use $ r $ and $ m $ to denote the rank and co-rank of $ M[P,Q] $ respectively.

	We say that $ [P,Q] $ has a $ k \times k $ \emph{square  at  i}, if there exists an $ i \in [m+r] $ such that $ P_{i} $ has exactly $ k $ more copies of East steps than $ Q_{i} $, and $ Q_{i} $ has exactly $ k $ more copies of North steps than $ P_{i} $. A $ k \times k $ square is \emph{proper} if $ i \in [k+1,m+r-k-1] $.\\

	In the definitions following, $ + $ denotes the concatenation of steps described. Thus  $ E+N = EN $. Hence, we can consider the paths $ P $ and $ Q $ as $ P = P_{i} + P_{i}^{'} $ and $ Q = Q_{i} +Q_{i}^{'} $, where $ P_{i}^{'} $ and $ Q_{i}^{'} $ are the last $ (m+r-i) $ steps of $ P $ and $ Q $ respectively. If $ X $ is any string, then $ r(X) = $ number of North steps in $ X $, and $ m(X) = $ number of East steps in $ X $.

	A lattice path presentation is said to have \emph{square-width}  $ k $ when the largest square it contains is a $ k \times k $ square. Square-width is closely associated with branch-width of a matroid, as we see below.
	
	\begin{lemma} \label{sqwidth}
		
		Let $ [P,Q] $ be a path presentation with a $ k \times k $ square. Then $ M[P,Q] $ has a $ U_{k,2k}$-minor.
	\end{lemma}
	
	\begin{proof}
		Let $ M[P,Q] $ be a minimal counter-example to our hypothesis, with $ P $ and $ Q $ being lattice paths from $ (0,0) $ to $ (m,r) $. Let the corners of the $ k \times k $ square be at $(i,j) $, $ (i+k,j) $, $ (i,j+k) $ and $ (i+k,j+k) $. If $ i>0 $, then the first element is not part of the $ k \times k $ square. This implies that $ M[P,Q]\backslash1 $ contains a  $ k \times k $ square, which in turn implies that $ M[P,Q]\backslash1 $ contains a $ U_{k,2k} $-minor. But then so would $ M[P,Q] $, which contradicts our assumption.  Thus $ i= 0 $. Similarly, $ j=0 $ as otherwise, $ M[P,Q]/1 $ would contain a $ k \times k $ square. 
		
		Now, if $ 2k < m+r $, the path presentation of either $ M[P,Q]\backslash m+r $ or $ M[P,Q]/m+r $ contains a $k \times k $ square , which again leads to a contradiction. Hence, $ 2k = m+r $. Also, since $ k \leq $ min$ \{m,r\} $, it follows that $ m = r = k $. Thus $ P $ and $ Q $ bound a $ k \times k $ square, and so $M[P,Q] $ is in fact isomorphic to $ U_{k,2k} $ and the proof is complete.		
	\end{proof}
	
	\begin{corollary}
		
		Let $ M = M[P,Q] $ be a lattice path matroid and assume that $bw(M) \leq k $. Then the square-width is less than $ \lceil 3k/2 \rceil$. 
		
	\end{corollary}
	
	\begin{proof}
		Assume that the square width $ j $ is $\lceil 3k/2 \rceil $. Then by Lemma \ref{sqwidth}, $ M[P,Q] $ has $ U_{j,2j} $ minor. But the branch-width of $ U_{j,2j} $ is $ \lceil 2j/3 \rceil + 1 $ (Exercise 14.2.5, \cite{book}), and hence the branch-width of $ U_{j,2j} $ is at least $ k+1 $. But this is a contradiction to our assumption that the branch-width of $ M[P,Q] $ is at most $ k $.
	\end{proof}
	
	It is not difficult to prove that a class of lattice-path matroids has bounded branch-width if and only if it has bounded square-width.

	\begin{definition} \label{pull}
		Let $ [P,Q] $ be a path presentation on $ [m+r] $ with a proper $ k \times k $ square at $ i $. We define two new lattice path presentations from $ [P,Q] $ as follows:
		Let \\
		\begin{equation*}
		\begin{split}
		B_{i}(P) & = P_{i}  + k  \text{ copies of North steps, and }\\
		B_{i}(Q) & = Q_{i} + k  \text{ copies of East step}s		
		\end{split}
		\end{equation*} 
		be the lattice paths $ [B_{i}(P),B_{i}(Q)] $, and let 	 
		\begin{equation*}
		\begin{split}
		T_{i}(P) & =  k \text{ copies of East steps} + P_{i}^{'}  \text{ and} \\
		T_{i}(Q) & =  k  \text{ copies of North steps } + Q_{i}^{'}  
		\end{split}
		\end{equation*}
		\\
		be the lattice paths $ [T_{i}(P),T_{i}(Q)] $.
		
	\end{definition}
	
	Note that $ M([B_{i}(P),B_{i}(Q)]) $ is a lattice path matroid on the ground set $ [i+k] $ but we relabel  $ M([T_{i}(P),T_{i}(Q)]) $ to be a lattice path matroid on the ground set $ [i-k,m+r] $, so as to retain the same elements as in the original matroid. This will prove beneficial in the gluing operation that follows.

	\begin{figure}[h]
		\centering
		\includegraphics[scale=0.4]{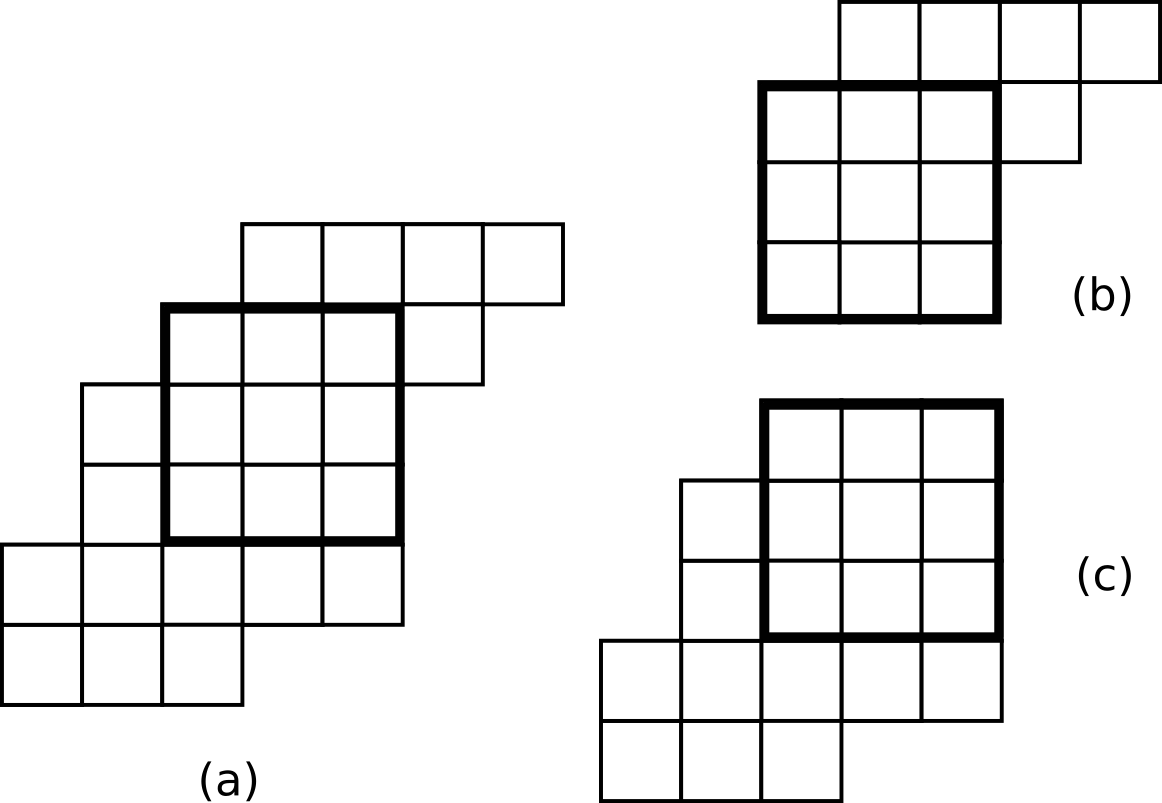} 
		\caption{(a) A proper $k \times k$ square in $ [P,Q] $ (b) $ [T_{i}(P),T_{i}(Q)] $ (c) $ [B_{i}(P),B_{i}(Q)] $}
		\label{top}
	\end{figure}
	
	In Figure \ref{top}, $ B_{7}(P) = P_{7}+NNN $ and $ B_{7}(Q) = Q_{7}+EEE $ and $ P_{7}^{'} = NNENEN $ and $ Q_{7}^{'} = ENEEEE $. Thus $ T_{7}(P) = $ $ EEE+P_{7}^{'} $  and 
	$ T_{7}(Q) = $ $ NNN +Q_{7}^{'} $.\\
	
	An intuitive property of these new path presentations is proved below:
	
	\begin{lemma} \label{bottom-minor}
		
		Let $ [B_{i}(P),B_{i}(Q)] $ and $ [T_{i}(P),T_{i}(Q)] $ be as in Definition \ref{pull}. Then $ [B_{i}(P),B_{i}(Q)], [T_{i}(P),T_{i}(Q)] $ are minors of $ [P,Q]$.
		
	\end{lemma}
	
	\begin{proof}
		
		We know that $ P = P_{i} + P_{i}^{'} $ and $ Q = Q_{i} + Q_{i}^{'} $, where $ 1 < i < m+r $. Then we know that 		
		$ r(P_{i}) + r(P_{i}^{'}) = r $, 
		$ r(Q_{i}) + r(Q_{i}^{'}) = r $, 
		$ m(P_{i}) + m(P_{i}^{'}) = m $, and 
		$ m(Q_{i}) + m(Q_{i}^{'}) = m $.
		
		Also, by definition of  paths $ P_{i} $ and $ Q_{i} $, 		
		$ m(P_{i}) - m(Q_{i}) = k $ and 
		$ r(Q_{i}) - r(P_{i}) = k $.
		
		Rearranging the above set of six equations, it is easy to see that 		
		$ r(P_{i}^{'}) - r(Q_{i}^{'}) = k $ and 
		$ m(Q_{i}^{'}) - m(P_{i}^{'}) = k $.
		
		These equations display the fact that $ Q_{i}^{'} $ has $ k $ more East steps than $ P_{i}^{'} $ and $ P_{i}^{'} $ has $ k $ more North steps than $ Q_{i}^{'} $. Since deletion requires the removal of East steps from both paths, and contraction requires removal of North steps from both paths, we can delete all the East steps from $ P_{i}^{'} $ and retain $ k $ East steps in $ Q_{i}^{'} $. Similarly, we can contract all the North steps from $ P_{i}^{'} $ and still retain $ k $ North steps in the same. 
		
		After these operations we obtain two paths $ P^{'} $ and $ Q^{'} $ that are of the form:		
		$ P^{'} = P_{i} + kN $ and
		$ Q^{'} = Q_{i} + kE $.
		
		These paths are nothing but $B_{i}(P) $ and $ B_{i}(Q) $.
		
		The same deductions as above imply that $ [T_{i}(P),T_{i}(Q)] \leq [P,Q] $.
	\end{proof}
	
	\begin{definition} \label{glue}
		If we have a pair of lattice path presentations $ [P_{B},Q_{B}] $ and $ [P_{T},Q_{T}] $ such that the following conditions are satisfied:
		
		\begin{enumerate}
			
			\item[(i)] The last $ k $ steps in $ P_{B} $ are North steps, i.e., $ P_{B} = P_{B}^{'} + k $ North steps
			\item[(ii)] The last $ k $ steps in $ Q_{B} $ are East steps, i.e., $ Q_{B} = Q_{B}^{'} + k $ East steps
			\item[(iii)] The first $ k $ steps in $ P_{T} $ are East steps, i.e., $ P_{T} = k $ East steps $ + P_{T}^{'}  $ 
			\item[(iv)] The first $ k $ steps in $ Q_{T} $ are North steps, i.e., $ Q_{T} = k $ North steps $ + Q_T^{'} $,
			
		\end{enumerate}
		
		then we define a gluing operation as $ GL([P_{B},Q_{B}],[P_{T},Q_{T}]) = [P,Q] $, where $ P = P_{B}^{'} + P_{T}^{'} $, and  $ Q =  Q_{B}^{'}  + Q_{T}^{'} $. \\
		
	\end{definition} 
	In other words, if we start with a lattice path $ [P,Q] $ and `pull them apart' at a $ k \times k $ square to give rise to two new lattice paths $ [B_{i}(P),B_{i}(Q)] $ and $ [T_{i}(P),T_{i}(Q)] $, then $ GL( [B_{i}(P),B_{i}(Q)]  ,[T_{i}(P),T_{i}(Q)]  ) $ will lead us back to the lattice path $ [P,Q] $ that we originally had. This fact is illustrated in the following lemma: 
	
	\begin{lemma} \label{same}
		
		Let $ [P,Q] $ be a path presentation with a $ k \times k $ square at $ i $, where $ P = P_{i} + P_{i}^{'} $ and $ Q = Q_{i} + Q_{i}^{'} $. Construct the two lattice paths  $ [B_{i}(P),B_{i}(Q)] $ and $ [T_{i}(P),T_{i}(Q)] $. Then $ GL( [B_{i}(P),B_{i}(Q)]  ,[T_{i}(P),T_{i}(Q)]  )  = [P,Q]$.
		
	\end{lemma}
	
	\begin{proof}
		By Definition \ref{pull}, paths $ B_{i}(P),B_{i}(Q),T_{i}(P)$ and  $ T_{i}(Q) $ satisfy conditions (i)  -  (iv)  in Definition \ref{glue}. Thus $ P_{B}^{'} = P_{i} $, $ Q_{B}^{'} = Q_{i} $, $ P_{T}^{'} = P_{i}^{'} $ and $ Q_{T}^{'} = Q_{i}^{'} $. Thus $ P_{B}^{'} + P_{T}^{'} = P_{i} + P_{i}^{'} = P $ and $ Q_{B}^{'}  + Q_{T}^{'} = Q_{i} +Q_{i}^{'} = Q $. This completes the proof.
	\end{proof}
	
	We are now well-equipped to prove the central lemma which proves pivotal in proving the main theorem.

	\begin{lemma} \label{imp}
		
		Let $ [P,Q] $ be a path presentation with a proper $ k\times k $ square at $ i $. Let $ [B_{i}(P),B_{i}(Q)]  $ and $ [T_{i}(P),T_{i}(Q)] $ be as in Definition \ref{pull}. Let $ [P_{B},Q_{B}] $ be a minor of $ [B_{i}(P),B_{i}(Q)]  $ with a $ k \times k $ square at the top and $ [P_{T},Q_{T}] $ be a minor of $ [T_{i}(P),T_{i}(Q)]  $, with a $ k \times k $ square at the bottom. Then $ GL([P_{B},Q_{B}],[P_{T},Q_{T}]) $ is a minor of $ [P,Q] $.
		
	\end{lemma}
	
	\begin{proof}
		
		We prove this by induction on $ n $, where $ n $ is the sum of size differences of $ \{ [P_{B},Q_{B}],[B_{i}(P),B_{i}(Q)] \} $ and $ \{ [P_{T},Q_{T}],[T_{i}(P),T_{i}(Q)] \}$. Note that $ [B_{i}(P),B_{i}(Q)]  $ has a $ k \times k $ square at the top and $ [T_{i}(P),T_{i}(Q)]  $, has a $ k \times k $ square at the bottom.
		
		When $ n =1 $, either $ \{ [P_{B},Q_{B}],[B_{i}(P),B_{i}(Q)] \} $ or $ \{ [P_{T},Q_{T}],[T_{i}(P),T_{i}(Q)] \}$ differ by a single element in the ground set. Let the former be true, in which case, an element, say $ e $, was deleted or contracted from $ [B_{i}(P),B_{i}(Q)] $ to obtain $ [P_{B},Q_{B}] $.
		
		If $ e \leq i $, then $ e $ is a step in the sub-strings $ P_{i} $ and $ Q_{i} $. We consider the different cases where $ e $ is a North step and an East step in the sub-strings $ P_{i} $ and $ Q_{i} $. 
		
		Case (i): $ e $ is an East step in both sub-strings
		
		We can decompose the paths as $ P = P_{0} + E + P_{1} + P_{i}^{'} $ and $ Q = Q_{0} + E + Q_{1} + Q_{i}^{'} $, where $ P_{0} + E + P_{1} = P_{i} $, $ Q_{0} + E + Q_{1} = Q_{i} $ and the East steps represent the $ i^{th} $ position. Then, it is obvious that we have to remove the East steps from both paths. Thus,
		\[ [B_{i}(P),B_{i}(Q)] = [P_{i} + kN, Q_{i} + kE ]
		= [P_{0} + E + P_{1} +kN, Q_{0} + E + Q_{1} + kE]. \]
		Therefore,
		\[ [B_{i}(P),B_{i}(Q)]\backslash e = [P_{0} + P_{1} +kN, Q_{0} + Q_{1} + kE ] = [P_{B},Q_{B}]. \]
		Also,
		\[ [P_{T},Q_{T}] = [T_{i}(P),T_{i}(Q)] = [kE + P_{i}^{'}, kN + Q_{i}^{'}]. \]
		Hence,
		\[ GL([P_{B},Q_{B}],[P_{T},Q_{T}]) = [P_{0} + P_{1} +P_{i}^{'}, Q_{0} + Q_{1} + Q_{i}^{'} ]. \]
		On the other hand, 
		\[ [P,Q] =  [P_{0} + E + P_{1} +P_{i}^{'}, Q_{0} + E + Q_{1} + Q_{i}^{'} ] \text{ and thus}\]
		\[ [P,Q]\backslash e =  [P_{0} + P_{1} +P_{i}^{'}, Q_{0} + Q_{1} + Q_{i}^{'} ] = GL([P_{B},Q_{B}],[P_{T},Q_{T}]). \]
		
		Case (ii): $ e $ is an East step in $ P $, but a North step in $ Q $
		
		The $ [B_{i}(P),B_{i}(Q)] $ in this case can be presented as :
		\[ [B_{i}(P),B_{i}(Q)] = [P_{i} + kN, Q_{i} + kE ]
		= [P_{0} + E + P_{1} +kN, Q_{0} + N + Q_{1} + kE]. \]
		Clearly, it is enough to remove the East step from $ P $. Now, $ Q = Q_{0} + N + Q_{1} + Q_{i}^{'} $. Then $ Q_{1} $ will contain an East step. Else, suppose that $ Q_{1} $ does not contain any East step. So we will have to remove an East step from $ Q_{i}^{'} $. Then, when we delete $ e $ from $ P $, we are essentially removing an East step from $ P_{i} $, but not $ Q_{i} $. But this is impossible as there is a $ k \times k $ square at $ i $ in both $ [B_{i}(P),B_{i}(Q)]  $ and $ [P_{B},Q_{B}] $  and hence, $ P_{i} $ has exactly $ k $ more East steps than $ Q_{i} $. Thus we will be able to find an East step in $ Q_{1} $ which can be removed. Then we rewrite $ Q_{i} $ as $ Q_{i} = Q_{0}^{'}  + E + Q_{1}^{'} $, where $ E $ is the first East step in $ Q_{1} $. Hence,
		
		\[ [B_{i}(P),B_{i}(Q)] = [P_{0} + E + P_{1} +kN, Q_{0}^{'} + E + Q_{1}^{'} + kE] \text{ and} \]
		\[ [B_{i}(P),B_{i}(Q)]\backslash e = [P_{0} + P_{1} +kN, Q_{0}^{'} + Q_{1}^{'} + kE]. \]
		Similar to the case above,
		\[ GL([P_{B},Q_{B}],P_{T},Q_{T}]) = [P_{0} + P_{1} +P_{i}^{'}, Q_{0}^{'} + Q_{1}^{'} + Q_{i}^{'} ]. \]
		Also,
		\[ [P,Q]\backslash e =  [P_{0} + P_{1} +P_{i}^{'}, Q_{0}^{'} + Q_{1}^{'} + Q_{i}^{'} ] = GL([P_{B},Q_{B}],[P_{T},Q_{T}]). \]
		
		Case (iii): $ e $ is a North step in $ P $, but an East step in $ Q $
		
		This case can be verified with very similar arguments as in Case (ii).
		
		Case (iv): $ e $ is a North step in both $ P $ and $ Q $
		
		The same conclusion follows by the arguments made in previous cases. \\
		From the above cases we observe that when $ e $ is deleted in $ [B_{i}(P),B_{i}(Q)] $, the changes are reflected in $ P_{i} $ and $ Q_{i} $ alone, as we have seen above. The same argument can be extended to the case of contraction of $ e $. Thus we conclude that when $ e \leq i $, $ GL([P_{B},Q_{B}],[P_{T},Q_{T}]) $ is a minor of $ [P,Q] $.\\

		Now, if $ e>i$, we claim that this has the same effect as deleting or contracting $ i $.
		
		Case (i) : $ e $ was deleted
		
		We know that deletion of an element involves removal of an East step from both paths.  The last $ k $ steps of  $ B_{i}(Q) $ are East steps alone, as $ [B_{i}(P),B_{i}(Q)] $ has a $ k \times k $ square at the top. But so does $ Q_{B} $, and thus we require that $ B_{i}(Q) $ has only East steps as the last $ k+1 $ steps, that is, the $ i^{th} $ element in $ B_{i}(Q) $ and consequently in $ Q_{i} $ is an East step. Hence deletion of any of the last $ k $ East steps from $ B_{i}(Q) $ is the same as deleting the  $ i^{th} $ East step. 
		
		Now, $ B_{i}(P) $ has North steps only as the last $ k $ steps. Then the $ i^{th} $ step can either be a North step or an East step. If it is an East step, then since deletion of an element in a lower bounding path is equivalent to deleting the East step at or before the corresponding element, deletion of any of the last $ k $ North steps will result in the removal of the $ i^{th} $ East step. The $ i^{th} $ step in $P_{i} $ is also an East step and thus deletion of $ e $ is the deletion of the $ i^{th} $ element. 
		
		But, if the $ i^{th} $ step is a North step, then deletion is the removal of the first East step that comes before the last $ k + 1 $ North steps. But this is the same as deletion of the $ i^{th} $ North step, both in  $ B_{i}(P) $ and $ P_{i} $. Hence, once again it is the deletion of the $ i^{th} $ element. Thus deleting any element $ e > i $ and then using the gluing operation to join the different path presentations together is the same as gluing them together and then deleting the $ i^{th} $ element. \\
		
		Case(ii) : $ e $ was contracted
		
		This case yields to a similar argument as case (i).\\
		
		So the result holds true when $ n = 1 $. Assume that it holds true for $ n = k $. We prove the result for $ n = k + 1 $. 
		
		Let $ B = M([B_{i}(P),B_{i}(Q)]) $, $ T = M([T_{i}(P),T_{i}(Q)]) $, $ B^{'} = M([P_{B},Q_{B}]) $ and $ T^{'} = M([P_{T},Q_{T}]) $. We then need to prove that $ GL(B^{'},T^{'}) $ is a minor of $ GL(B,T) $. Since $ B^{'} $, $ T^{'} $ are minors of $ B $, $ T $ respectively, $ B^{'} = B\backslash I/J $, for disjoint sets $ I $ and $ J $. Also, $ T^{'} = T\backslash I^{'}/J^{'} $ for disjoint sets $ I^{'} $ and $ J^{'} $. Since \[ (|B|-|B^{'}|) + (|T| - |T^{'}|) = k+1, \] one of $ I,J,I^{'},J^{'} $ is non-empty. Without loss of generality, let $ I $ be non-empty as the cases when $ J $, $ I^{'} $ or $ J^{'} $ are non-empty are identical. 
		
		Let $ e \in I $. Then \[ (|B|-|B\backslash e |) + (|T| - |T|) = 1 \] and hence by the case proved for when $ n = 1 $, $ GL(B\backslash e,T) $ is a minor of $ GL(B,T) $. Also, \[ (|B\backslash e|-|B^{'}|) + (|T| - |T^{'}|) < k+1. \] Thus, by our induction hypothesis, $ GL(B^{'},T^{'}) $ is a minor of $ GL(B\backslash e,T) $ which is already a minor of $ GL(B,T) $. This finishes the proof for the case $ n = k + 1  $ and completes our induction argument. 
	\end{proof}
	
	\section{Proof of the main theorem}
	
	With the aid of the above lemma, we are now ready to prove that the class of lattice path matroids with bounded square-width is well-quasi-ordered. We use the $ minimal $ $ bad $ $ sequence $ argument in the proof of the same. An infinite sequence $ a_{1},a_{2},a_{3},\ldots $ is $ bad $ if there does not exist an $ i $ and $ j $ such that $ a_{i} \leq a_{j}$. Otherwise, the sequence is good. Also, a subsequence $ (a_{i_{1}},a_{i_{2}},\ldots) $ is bad if it is a bad sequence. In essence, a graph class or matroid class is well-quasi-ordered if and only if it does not have a bad sequence, as graph and matroid classes do not contain infinite decreasing sequences.
	We require a lemma about bad sequences to explain the minimal bad sequence argument.
	
	\begin{lemma} \label{chain}
		
		Let $ a_{1},a_{2},\ldots,$ be an infinite sequence with no bad subsequences. Then there exists $ i_{1}<i_{2}<i_{3}<\ldots$, such that $ a_{i_{s}} \leq a_{i_{s+1}} $ $ \forall $ $ s $.
		
	\end{lemma}
	
	\begin{proof}
		
		We begin by constructing a directed graph as follows: if $ a_{i_{s}} < a_{i_{t}} $ and there does not exist a $ k $ such that $ s,t \neq k $ and $ a_{i_{s}} \leq a_{i_{k}}  \leq a_{i_{t}} $, then we have a directed edge from vertex $ a_{i_{s}} $ to vertex $ a_{i_{t}} $. 
		
		Now, $ G $ has to be a directed graph with finitely many connected components. Else, suppose that $ G $ has infinitely many connected components $ G_{1},G_{2},\ldots $. Then selecting a vertex from each component provides us with infinitely many vertices $ v_{1},v_{2},\ldots $ that correspond to $ a_{i_{1}},a_{i_{2}} ,\ldots $ such that there does not exist $ s,t $ where $ a_{i_{s}} \leq a_{i_{t}} $, which contradicts our assumption that $ a_{1},a_{2},\ldots,$ does not have a bad subsequence.
		
		Clearly, at least one among the finite components of $ G $ must have infinite number of vertices. Hence, by K\"onig's Lemma \cite{konig}(see also \cite{Franchella1997}), this infinite graph either contains a vertex of infinite degree or an infinite simple path. If there exists such a vertex, then the matroids corresponding to the adjacent infinite number of vertices form a bad sequence. Thus $  G $ cannot contain a vertex of infinite degree. Hence it contains a simple path which completes our proof.
	\end{proof}

	For an arbitrary set $ \boldsymbol{\Sigma} $, let $ \boldsymbol{\Sigma}^{*} $ be the set of all finite sequences of $ \boldsymbol{\Sigma} $. Any quasi-order $ \leq $ on $ \boldsymbol{\Sigma} $ defines a quasi-order $ \preceq $ on $ \boldsymbol{\Sigma}^{*} $ as follows: $ (a_{1},a_{2},\ldots,a_{m}) $ $ \preceq $ $ (b_{1},b_{2},\ldots,b_{n}) $ if and only if there is an order-preserving injection $ f : \{ a_{1},\ldots,a_{m}\} \to \{b_{1},\ldots,b_{n}\} $ with $ a_{i} \leq f(a_{i}) $ for each $ i $. Then  Higman's Lemma \cite{PLMS:PLMS0326}
	states that $ (\boldsymbol{\Sigma}^{*},\preceq )$ is a well-quasi-order if $ (\boldsymbol{\Sigma},\leq)  $ is a well-quasi-order.

	\begin{theorem}
		
		Let $ \mathcal{L}_{k} $ be the class of lattice path matroids with square-width at most $ k $. Then, $ \mathcal{L}_{k} $ is well-quasi-ordered. 		
		
	\end{theorem}
	
	\begin{proof}
		We prove this by induction on square-width of the class.
		When $ n = 0  $, the class consists of path presentations with $ P = Q $. Thus they can be represented by a combination of horizontal lines that go right or vertical lines that go up. Thus the corresponding matroids consist of only loops and co-loops. Suppose $ l $ counts the number of loops in the matroid and $ c $ counts the number of co-loops. Then each matroid in $ \mathcal{L}_{0}  $ can be expressed as  an ordered pair $ (l,c) $. If $ M $ corresponds to $ (l,c) $, and $ M^{'} $ corresponds to $ (l^{'},c^{'}) $, then $ M $ is a minor of $ M^{'} $ if and only if $ l \leq l^{'} $ and $ c \leq c^{'} $. But this is a word in the alphabet of integers and is well-quasi-ordered by Higman's Lemma.

		We assume that the result is true when $ n = k $, i.e., $ \mathcal{L}_{k}$ is well-quasi-ordered. We now prove that $ \mathcal{L}_{k+1} $ is well-quasi-ordered.
		
		Suppose the contrary that $ \mathcal{L}_{k+1} $ is not well-quasi-ordered. Then there exists a bad sequence in $ \mathcal{L}_{k+1} $. Clearly, every bad sequence has only finitely many path presentations that belong to $ \mathcal{L}_{k} $. Or else, if the bad sequence contains infinitely many path presentations from $ \mathcal{L}_{k} $, then this sub-sequence of matroids in $ \mathcal{L}_{k} $ is an infinite bad sequence in itself, which contradicts our inductive assumption. Thus we can safely remove this finite sub-sequence from the bad sequence without altering the property of being bad. As the class of nested matroids is well-quasi-ordered 
		bad sequences can contain at most finitely many presentations that correspond to nested matroids. These matroids can also be removed from the corresponding bad sequences without much ado. Hence, from now on we only consider bad sequences that are made up entirely of path presentations of square width $ k+1 $, that are not nested. It is easy to prove that uniform matroids are well-quasi-ordered, and hence it would have sufficed to remove rectangular path presentations alone, as they correspond to uniform matroids, instead of nested ones.
		
		We now construct a minimal bad sequence as follows: Assume we have chosen $ L_{1},\ldots,L_{i-1} $ to be the initial segment of our minimal bad sequence. Then, in all the bad sequences that start with $ L_{1},\ldots,L_{i-1} $, consider the smallest path presentation in position $ i $ to be at the $ i^{th} $ position of the minimal bad sequence. Thus the first presentation in the minimal bad sequence is the smallest that can start such a sequence, $ L_{2} $ is the smallest  in position $ 2 $ in all bad sequences that start with $ (L_{1}) $, $ L_{3} $ is the smallest in position $ 3 $ in all bad sequences that start with $ (L_{1}, L_{2}) $ and so on. We denote this sequence by $ L_{1}, L_{2}, L_{3}, \ldots$, where $ L_{i} = [P_{i},Q_{i}] $. It can be seen easily that $ L_{1}, L_{2}, L_{3}, \ldots$ is a bad sequence in itself. If not, there exists $ i<j $ such that $ L_{i} \leq L_{j} $. By virtue of construction of the minimal bad sequence, $ L_{j} $ is the smallest presentation in the $ j^{th} $ position among all bad sequences that start with  $ L_{1},\ldots ,L_{j-1}$. Thus $ L_{i} $ is a minor of $ L_{j} $ in that particular sequence that $ L_{j} $ is chosen from, which contradicts the fact that it is a bad sequence. Since all $ L_{i}$s belong to $ \mathcal{L}_{k+1} $ and we do not have any nested matroids, they all have a proper $ k+1 $ square at say, $ j(i) $. 
		
		We apply the pulling apart operation as defined in Definition \ref{pull} to the sequence $ L_{1}, L_{2}, L_{3}, \ldots$ to obtain two new sequences $ B_{1},B_{2},B_{3},\ldots $, where $ B_{i} = [B_{j(i)}(P_{i}),B_{j(i)}(Q_{i})] $ and $ T_{1},T_{2},T_{3},\ldots $, where $T_{i} = [T_{j(i)}(P_{i}),T_{j(i)}(Q_{i})] $. Since the sequence $ L_{1}, L_{2}, L_{3}, \ldots$ is the minimal bad sequence, $ B_{1},B_{2},B_{3},\ldots $ cannot contain a bad subsequence. This can be seen as follows : let there exist a bad subsequence of $ B_{1},B_{2},B_{3},\ldots $, say $ B_{i_{1}},B_{i_{2}},B_{i_{3}},\ldots $, then $ L_{1}, L_{2},\ldots,L_{{i}_1-1},$ $B_{i_{1}},B_{i_{2}},\ldots $ is a bad sequence. If it were not a bad sequence, then there would exist $ L_{k} $ and $ B_{i_{j}} $ such that $ L_{k} \leq B_{i_{j}} $. But by Lemma \ref{bottom-minor}, $ L_{k}  \leq B_{i_{j}} \leq L_{i_{j}}$. Now, $ B_{i_{1}} $ is smaller than $ L_{i_{1}} $ which contradicts the fact that $ L_{1}, L_{2}, L_{3},\ldots$ is the minimal bad sequence.
		
		Thus in $ B_{1},B_{2},B_{3},\ldots $, there exists no bad subsequence. Hence by Lemma \ref{chain}, there exists a sub-sequence $ i_{1} < i_{2} < i_{3} < \ldots $ such that $ B_{i_{1}} \leq B_{i_{2}} \leq B_{i_{3}} \leq \ldots $. This implies that for some $ s < t $, $ T_{i_{s}} \leq T_{i_{t}} $, or else $ T_{1},T_{2},T_{3},\ldots $ would be a bad sequence (by the same reasoning as before). Thus $ B_{i_{s}} \leq B_{i_{t}} $ and $ T_{i_{s}} \leq T_{i_{t}} $. By Lemma \ref{imp}, $ GL(B_{i_{s}},T_{i_{s}}) = L_{i_{s}} $ is a minor of $ GL(B_{i_{t}},T_{i_{t}}) = L_{i_{t}} $. This is a contradiction to our assumption that $ L_{1},L_{2},L_{3},\ldots $ is a bad sequence. Thus $ L_{k+1} $ is well-quasi-ordered.
	\end{proof}

	\end{document}